\documentclass{llncs}
 \usepackage{mystyle}
\begin{document}
\title{The $p$-adic Integers as Final Coalgebra }
\titlerunning{The $p$-adic Integers as Final Coalgebra}
\author{Prasit Bhattacharya} 
\institute{Department of Mathematics, Indiana University, \\
Bloomington, IN 47405 USA}
\authorrunning{Prasit Bhattacharya}

\maketitle
  
\begin{abstract} 
 We express the classical $p$-adic integers $\pAdic$, as a metric space, as the final coalgebra to a certain endofunctor. We realize the addition and the multiplication  on $\pAdic$ as  the coalgebra maps from $\pAdic \times \pAdic$.
\end{abstract} 
\section{Introduction}
 The set of $p$-adic integers $\pAdic$, for a prime $p$, has been  of mathematicians' interest for centuries. It is also known that the $p$-adic integers can be visualized as fractals (see \cite{C}). In recent past Leinster (\cite{L}) showed that many fractal like objects, often called self similar spaces can be obtained as final coalgebra for certain functors. Recall that,
\begin{defn} Given an endofunctor $F: \mathcal{C} \to \mathcal{C}$, an $F$-algebra is a pair $(X,f)$ such that $X$ is an object of $\mathcal{C}$ and $f:F(X) \to X$. Similarly $F$-coalgebra is a pair $(Y,g)$ such that $Y$ is an object of $\mathcal{C}$ and $f: Y \to F(Y)$. The initial object in the category of $F$-algebra is known as \emph{initial $F$-algebra} and final object in the category of $F$-coalgebra is the \emph{final $F$-coalgebra}.  
\end{defn} 

Initial algebra results are very important due to the connection of initiality and recursion.  In recent years, final coalgebra results are also important.  The logic-related importance comes from connections to circularity, streams, non-wellfounded sets.  It also has a lot of computer science connections: processes that run forever, bisimulation, and related concepts. Often it is found that one can take an area of classical mathematics and put it under the same roof as all the other final coalgebra results. Many historically important mathematical objects have been realized as the initial algebra or the final coalgebra for certain functors. There is a plethora of such examples: natural numbers, infinite-binary trees, the unit interval, Serpenski Gasket, streams and many more. A wide variety of such examples are discussed in \cite{AMM,R}. The purpose of the paper is to establish that the $p$-adic integers $\pAdic$ is the final coalgebra as a metric space. We do not claim to find any new property or result about the $p$-adic integers. Rather the method in this paper provides a new universal characterization of the $p$-adic integers.

Let us pause for a moment to recall the basics of $p$-adic integers. As a ring $\pAdic$ is obtained as the inverse limit of the diagram 
\[ \ldots \to \mathbb{Z}/p^n \to \ldots \to \mathbb{Z}/p^2 \to \mathbb{Z}/p, \]
where the map $q_n:\mathbb{Z}/p^{n+1} \to \mathbb{Z}/p^{n}$ is the canonical quotient map obtained by reduction modulo $p^n$. Each element  $a \in \pAdic$ can be expressed as an infinite stream
\[ a = (\ldots, a_2, a_1,a_0)\]
where each $a_i \in \lbrace0, 1, \ldots, p-1\rbrace $. Such a stream should be thought of as the power series $\sum_{i = 0}^{\infty} a_ip^i$ or more specifically the inverse limit of the elements 
\[ \tau_{n}(a)=\sum_{i = 0}^{n} a_ip^i \in \mathbb{Z}/p^{n+1}.\]
This makes sense as $q_n(\tau_n(a)) = \tau_{n-1}(a)$. Moreover, $\pAdic$ is a metric space with the $p$-adic distance 
\[ d_p(a,b) = p^{-i}\]
if $i$ is the least number such that $p^{i} \mid a-b$. Representing the $p$-adic integers as streams we see that $d(a,b)= p^i$ if $a_k =b_k$ for $k\leq i$ but $a_{k+1} \neq b_{k+1}$. This metric on $\pAdic$ is in fact an ultrametric, i.e. the metric satisfies a stronger triangle inequality 
\[d_p(a,c) \leq max \lbrace d_p(a,b), d_p(b,c)\rbrace.\]
Another fact that is worth mentioning for the purpose of the paper is that the natural numbers $\mathbb{N}$ with the $p$-adic metric is embedded inside $\pAdic$ and its image is a dense subset. To see this notice that the image of $n \in \mathbb{N}$ is $(\ldots,0,0, a_k, \ldots, a_0)$ if $n = \sum_{i=0}^k a_ip^i$ and hence  $\mathbb{N}$ sits inside $\pAdic$ as the set of streams of finite length. The $p$-adic integers form a ring, i.e. they have addition and multiplication. If we restrict the addition and multiplication to $\mathbb{N}$ we get the usual addition and multiplication of $\mathbb{N}$. Though $p$-adic integers look like a power series, its addition and multiplication are very different from the addition and multiplication of power series as there is a notion of `carry over'. For details about the properties of the $\pAdic$ reader may refer to \cite{Serre}.  
\begin{notn}Let $\Ult$ be the category of $1$-bounded ultrametric spaces, i.e. ultra-metric spaces with diameter at most $1$, and morphisms are  nonexpanding maps, i.e. a morphism $f:P \to Q$ satisfies \[d_{Q}(f(x), f(y)) \leq d_{P}(x,y).\] We will denote $\Ult_*$ for the category of pointed $1$-bounded ultrametric spaces. $\cUlt$ and $\cUlt_*$ will denote the category of $1$-bounded Cauchy complete ultrametric spaces and the category of pointed $1$-bounded Cauchy complete ultrametric spaces respectively. \end{notn}
Now we describe the results that are obtained in this paper. Consider the functor 
\[\mathcal{F}_p: \Ult \to \Ult\]
 which maps $X \mapsto \frac{1}{p}X \times V_p$, where $V_p = \lbrace 0, \ldots,p-1 \rbrace$ with discrete metric and $\frac{1}{p}X$ represents the metric space obtained by contracting $X$ by a factor of $p$. The main result in Section~\ref{sec:initialfinal} is the following.
\begin{main} \label{main:final} The underlying metric space of $\pAdic$ denoted by $Z_p$, is the final $\mathcal{F}_p$-coalgebra in $\Ult$ with the coalgebra map $$\phi: Z_p \to  \frac{1}{p}Z_p \times V_p $$
where $\phi(a)= ((\ldots, a_2, a_1), a_0)$.
\end{main}
In Section~\ref{sec:addmult}, we give $Z_p \times Z_p$ two different coalgebra structures by explicitly defining the two maps $$\tilde{A},\tilde{M}: Z_p \times Z_p \to  \frac{1}{p}(Z_p \times Z_p)\times V_p $$ in such a way that we have the following two results.
\begin{main} \label{main:add}The final $\mathcal{F}_p$-coalgebra map $\alpha: (Z_p \times Z_p, \tilde{A}) \to (Z_p, \phi)$ is the $p$-adic addition. 
\end{main}
\begin{main} \label{main:mult}The final $\mathcal{F}_p$-coalgebra map $\mu: (Z_p \times Z_p, \tilde{M}) \to (Z_p, \phi)$ is the $p$-adic multiplication. 
\end{main}
\subsection*{Acknowledgement} I would like to thank Larry Moss introducing me to this subject and assisting me at various stages of this project, Robert Rose for various discussions about this project and finally  Michael Mandell for being very supportive as an adviser and allowing me the freedom to explore different areas of Mathematics.

\section{ The Final $\mathcal{F}_p$-coalgebra} \label{sec:initialfinal}

In this section we show that the underlying metric space of $\pAdic$, which we denote by $Z_p$, is the final $\mathcal{F}_p$-coalgebra. The following results provide techniques to obtain initial algebras and final coalgebras. 
\begin{thm}{\emph{(J.Ad\'amek)}} \label{thm:ada} Let $ \mathcal{C}$ be a category with the initial object $\perp$ and $F$ be an endofunctor on $\mathcal{C}$. Suppose  that the colimit of the initial $\omega$-chain, i.e.
\[ \perp \overset{\iota}{\to} F(\perp) \overset{F(\iota)}{\to} F^2(\perp) \overset{F^2(\iota)}{\to} \ldots\]
exists and $F$ preserves the colimit, then the initial $F$-algebra is the object
\[ \mu F = \colim_{\omega}F^n(\perp).\]
\end{thm}
The dual version of the above theorem is due to M.Barr (see \cite{Barr, AMM}), and it helps us to produce the final $F$-coalgebra under suitable conditions. 
\begin{thm}{\emph{(M.Barr)}} \label{thm:barr}Let $ \mathcal{C}$ be a category with the final object $\top$ and $F$ be an endofunctor on $\mathcal{C}$. Suppose  that the colimit of the final $\omega$-chain, i.e.
\[\ldots  \overset{F^2(\tau)}{\to} F^2(\top) \overset{F(\tau)}{\to} F(\tau) \overset{\tau}{\to} \top \]
exists and $F$ preserves the colimit, then the final $F$-coalgebra is the object
\[ \nu F = \inlim_{\omega} F^n(\top).\]
\end{thm} 
 \begin{rem} \label{rem:cartesian}The metric of the cartesian product of two objects in $\Ult$ is different from the metric of the cartesian product in the category of metric spaces. To be precise the cartesian product in $\Ult$ for two objects $(X, d_X)$ and $(Y,d_Y)$, is the ultra-metric space $(X \times Y, d_{X \times Y})$, where 
\[ d_{X \times Y}((x_1,y_1), (x_2, y_2)) = max \lbrace d_X(x_1, x_2), d_Y(y_1, y_2)\rbrace.\]
Thus one can check that in the infinite cartesian product 
\[ V =  \ldots \times \frac{1}{p^n}V_p\times\ldots \times \frac{1}{p^2}V_p \times \frac{1}{p}V_p \times V_p\]
the distance function is given by the formula 
\[ d_V((\ldots, a_2, a_1, a_0), (\ldots, b_2, b_1, b_0)) = p^{-i}\]
where $i$ is the smallest number such that $a_i \neq b_i.$ Thus $V \iso Z_p$ is an ultra-metric space. 
\end{rem}
\begin{proof}[Proof of Main Theorem \ref{main:final}]
Note that the category $\Ult$ has a final object $\top$, the one point set. First we notice that by forgetting the metric structure, i.e. $\mathcal{F}_p$ as an endofunctor on sets, is the functor that sends $X \mapsto X \times V_p$. By using  Theorem~\ref{thm:barr} or otherwise one can conclude that the final $\mathcal{F}_p$-coalgebra on sets is the infinite cartesian product of $V_p$, which is isomorphic to the underlying set of $Z_p$ (see Remark~\ref{rem:cartesian} for further clarification). We will show that $Z_p$ is indeed the final $\mathcal{F}_p$-coalgebra in $\Ult$. 

So assume that $(X, f)$ is an $\mathcal{F}_p$-coalgebra in $\Ult$ where $f$ is a nonexpanding map. Since the underlying set for $Z_p$ is final, there exists a unique set map 
\[ g: X \to Z_p\]
with the diagram
 \begin{equation}\label{eqn:comp}\xymatrix{
X \ar[d]_g\ar[r] &  \frac{1}{p} X \ar[d]^{\mathcal{F}_p(g)} \times V_p\\
Z_p \ar[r]_{\phi} & \frac{1}{p} X \times V_p
} \end{equation}
commuting. Thus one can check that $g(x) = (\ldots , a_2, a_1, a_0)$ if the composite 
\begin{equation} X \overset{f}{\to} \frac{1}{p} X \times V_p \overset{\mathcal{F}_p(f)}{\to} \ldots  \overset{\mathcal{F}_p^{\circ n}(f)}{\to}\frac{1}{p^{n+1}} X \times \frac{1}{p^{n}} V_p \times \ldots \times V_p  \end{equation}
sends $x \mapsto (x', (a_n, \ldots, a_0))$. 
\begin{clm} The map $g$ as a map of metric spaces is nonexpanding i.e.  
\[ d_{Z_p}(g(x_1), g(x_2)) \leq d_{X}(x_1,x_2),\]
 hence a morphism in $\Ult$. 
\end{clm}
Notice that the map $f$ and $\mathcal{F}_p^{\circ n}(f)$ are all nonexpanding maps. Thus the composite of the maps in the Diagram~\ref{eqn:comp}, further composed with the projection map onto $ \frac{1}{p^{n}} V_p \times \ldots\times \frac{1}{p}V_p \times V_p$, is nonexpanding. This composite can also be regarded as $\pi_n \circ g$ where 
\[ \pi_n: Z_p \to  \tau_{\leq n}(Z_p)= \frac{1}{p^{n}} V_p \times \ldots\times \frac{1}{p}V_p \times V_p\]
is the projection of $Z_p$ onto its first $n+1$-coordinates. Thus $\pi_n \circ g$ is $1$-bounded, i.e.
\[d_{\tau_{\leq n}(Z_p)}(\pi_n(g(x_1)),\pi_n(g(x_2))) \leq d_{X}(x_1,x_2). \]  For the map $g: X \to Z_p$ we have the equality
\[ d_{Z_p}(g(x_1), g(x_2)) = max \lbrace d_{\tau_{\leq n}(Z_p)}(\pi_n(g(x_1)),\pi_n(g(x_2))): n \in \mathbb{N} \rbrace.\]
Thus we get $d_{Z_p}(g(x_1), g(x_2)) \leq d_{X}(x_1,x_2)$, i.e. $g$ is a nonexpanding morphism. Moreover, $g$ is unique in $\Ult$, as it is unique as a set  map. Hence $Z_p$ is the final $\mathcal{F}_p$-coalgebra in $\Ult$.
 \end{proof}
\begin{rem} Having established $Z_p$ as the final $\mathcal{F}_p$-coalgebra reader may be curious to know the initial $\mathcal{F}_p$-algebra. The category $\Ult$ does not have any initial $\mathcal{F}_p$-algebra. However, if we consider $\mathcal{F}_p$ as an endofunctor of $\Ult_*$, which has an initial object $\perp$, one can use Theorem~\ref{thm:ada} to see that the initial $\mathcal{F}_p$-algebra is the colimit of the diagram
\[ \perp \to \perp \times V_p \to \perp \times \frac{1}{p}V_p \times V_p \to \ldots \to \perp \times \frac{1}{p^n}V_p \times \ldots \frac{1}{p}V_p \times V_p \to \ldots  \]
which is the set of finite streams in $V_p$. The argument is similar to the proof of Theorem~\ref{main:final}, hence left to the reader. In fact this set can be identified with the set of natural numbers $\mathbb{N}$ with the $p$-adic metric under the map $(a_n, \ldots, a_0) \mapsto \sum_{i=0}^n a_ip^i$. If we work in $\cUlt_*$, then we see that the initial $\mathcal{F}_p$-algebra is the Cauchy completion of $\mathbb{N}$ with the $p$-adic metric, which is precisely $Z_p$.  
\end{rem}

\section{The Addition and the Multiplication in $\pAdic$} \label{sec:addmult}
The addition and the multiplication in $\pAdic$ have the notion of `carry over', hence are different from the addition and multiplication of power series. It is convenient to establish some notation before we give explicit formula for the $p$-adic addition and multiplication. 
\begin{notn} Fix a prime $p$. For any integer $n$, define $[n]_p$ to be the number in $\lbrace 0, \ldots, p-1\rbrace$, which is in the congruence class of  $n$ modulo $p$. Define $k_p(n)$ to be the unique number such that we have the formula 
\[ n = k_p(n)p + [n]_p.\]
\end{notn}
\begin{defn} The $p$-adic addition is a map $\alpha: \pAdic \times \pAdic \to \pAdic$, such that $$\alpha(a, b) = (\ldots , [\alpha_2(a,b)]_p,[\alpha_1(a,b)]_p, [\alpha_0(a,b)]_p),$$ where 
\[ \alpha_i: \pAdic \times \pAdic \to \mathbb{N}\]
is given by the inductive formula 
\begin{enumerate}
\item $\alpha_0(a,b) = a_0 + b_0$ and 
\item $\alpha_i(a,b) = a_i + b_i + k_p(\alpha_{i-1}(a,b))$.
\end{enumerate}
\end{defn}
\begin{defn} The $p$-adic multiplication is a map $\mu: \pAdic \times \pAdic \to \pAdic$, such that $$\mu(a, b) = (\ldots , [\mu_2(a,b)]_p,[\mu_1(a,b)]_p, [\mu_0(a,b)]_p),$$ where 
\[ \mu_i: \pAdic \times \pAdic \to \mathbb{N}\]
is given by the inductive formula 
\begin{enumerate}
\item $\mu_0(a,b) = a_0b_0$ and 
\item $\mu_i(a,b) = \sum_{i =0}^{n}a_ib_{n-i} + k_p(\mu_{i-1}(a,b))$.
\end{enumerate}
\end{defn}
The key property of the addition and the multiplication of $\pAdic$ is that when restricted to the image of $\mathbb{N}$, they are the usual addition and the multiplication of $\mathbb{N}$ respectively. Since $\mathbb{N}$ is dense in $\pAdic$ it is enough to check that the addition and the multiplication on $\mathbb{N}$ are nonexpanding with the $p$-adic metric on $\mathbb{N}$, to conclude that $\alpha$ and $\mu$ are nonexpanding maps. Recall that the $p$-adic metric on $\mathbb{N}$ is given by 
\[ d_p(m,n) = p^{-v(m-n)}\]
where $v(k)$ is the valuation, i.e. the maximum power of $p$ which divides $k$. Since the valuation function satisfies 
\[ v(m+n) \geq max \lbrace v(m), v(n)\rbrace\]
and 
\[ v(mn) \geq max \lbrace v(m), v(n)\rbrace,\]
the addition and the multiplication maps on $\mathbb{N}$ are nonexpanding.

For convenience, let's denote the natural number with the $p$-adic metric by $N_p$. Notice that the coalgebra map on $Z_p$ restricted to $N_p$ gives a coalgebra map on $N_p$ 
\[ \phi_{N_p}: N_p \to \frac{1}{p}N_p \times V_p\]
where $n \mapsto ([n]_p, k_p(n))$. In order to get the addition and multiplication on $\pAdic$ as $\mathcal{F}_p$-coalgebra maps, we define two different coalgebra structures on $N_p$, i.e. maps  
\[A, M: N_p \times N_p \to \frac{1}{p}(N_p \times N_p) \times V_p  \]
such that the addition on natural numbers is the $\mathcal{F}_p$-coalgebra map from $(N_p \times N_p, a) \to (N_p, \phi_{N_p})$ and the multiplication on natural numbers is the $\mathcal{F}_p$-coalgebra map from $(N_p \times N_p,A) \to (N_p, \phi_{N_p})$. Let $\tilde{A}$ and $\tilde{M}$ be the Cauchy completion of $A$ and $M$ respectively. It is not hard to see from the properties of $A$ and $M$ that the final maps $(Z_p \times Z_p, \tilde{A}) \to (Z_p, \phi)$ and $(Z_p \times Z_p, \tilde{M}) \to (Z_p, \phi)$ are the addition map $\alpha$ and the multiplication map $\mu$, respectively. 
\subsection{The addition in $\pAdic$}
Let $A: N_p \times N_p \to \frac{1}{p}(N_p \times N_p) \times V_p$ be the map given by the formula 
\[ A(m,n)= (A_2(m,n), A_1(m,n), A_0(m,n)) = (k_p(m) + k_p([m]_p + [n]_p) ,k_p(n),[m+n]_p).\] 
\begin{prop} The map $A$ is a nonexpanding map.
\end{prop}
 \begin{proof} If $d_p((m_1,n_1),(m_2,n_2)) =p^0 = 1$ then there is nothing to check as all the metric spaces are $1$-bounded. If $d_p((m_1,n_1),(m_2,n_2)) =p^{-i}$, where $i\geq 1$, we need to show that 
\begin{enumerate}
\item $d_{V_p}(A_0(m_1,n_1), A_0(m_2,n_2)) \leq p^{-i}$ and 
\item for $k \in \lbrace0,1,2 \rbrace$, $d_{N_p}(A_k(m_1,n_1), A_k(m_2,n_2)) \leq p^{-i+1}$ in $N_p$ as we are contracting the image of these maps by a factor of $\frac{1}{p}$. 
\end{enumerate}
Observe that $p^i\mid m_1 -m_2$ and $p^i \mid n_1-n_2$. In particular, we see that  
\begin{eqnarray*}
A_0(m_1,n_1) -A_0(m_2, n_2) &=& [m_1 + n_1]_p - [m_1 + n_2]_p \\
&=& [m_1 + n_1 - m_2 + n_2]_p \\
& =& 0
\end{eqnarray*}
Thus $d_{V_p}(A_0(m_1,n_1), A_0(m_2,n_2))=0$. Next we will show that 
\[ p^{i-1} \mid A_k(m_1,n_1) - A_k(m_2, n_2)\]
for $k \in \lbrace1,2\rbrace$. Notice that if $p^{i} \mid m_1-m_2$, then $[m_1]_p = [m_2]_p$. As a result 
\[ m_1 - m_2 = (k_p(m_1) -k_p(m_2))p\]
which means that $p^{i-1} \mid (k_p(m_1) -k_p(m_2))$. Similar arguments show that $p^{i-1} \mid (k_p(n_1) -k_p(n_2))$, thus the map $A_2$ satisfies the desired condition. Using the fact that $[m_1]_p = [m_2]_p$ and $[n_1]_p = [n_2]_p$, we get $[m_1]_p + [n_1]_p = [m_2]_p + [n_2]_p $, hence $$k_p([m_1]_p + [n_1]_p) = k_p([m_2]_p + [n_2]_p)$$. This concludes the result for the function $A_1$.  
\end{proof}
Before we prove the main result, we need to prove a short technical result for the function $k_p$.  
\begin{lem} \label{lem:tech} The following formula holds: 
\[ k_p(mp + n) = m + k_p(n)\]
\end{lem}
Proof is left to the reader to verify. 
\begin{proof}[Proof of Main~Theorem~\ref{main:add}] We need to verify that the diagram 
\begin{equation} \label{diag:add}
\xymatrix{
N_p \times N_p \ar[d]_{+}\ar[r]^{A} & \frac{1}{p}(N_p \times N_p) \times V_p \ar[d]^{\mathcal{F}_p(+)}\\
N_p \ar[r]_{\phi_{N_p}} & \frac{1}{p}N_p \times V_p
}
\end{equation}
commutes. 
So we need to show that 
\[(k_p(n+m), [m+n]_p) = (A_1(m,n)+ A_2(m,n),[m]_p + [n]_p).\]
Clearly, 
\begin{eqnarray*}
[m+n]_p &=& [m]_p + [n]_p.
\end{eqnarray*}
Further,
\begin{eqnarray*}
k_p(m+n) &=& k_p(k_p(m)p +[m]_p + k_p(n)p + [n]_p)\\
&=& k_p((k_p(m) + k_p(n))p + [m]_p + [n]_p) \\
&=& k_p(m) + k_p(n) + k_p([m]_p + [n]_p) \hspace{10pt} \text{by Lemma~\ref{lem:tech}}\\
&=& A_1(m,n)+ A_2(m,n)
\end{eqnarray*}
Now applying the Cauchy completion functor to the Diagram~\ref{diag:add} we get, 
\begin{equation}
\xymatrix{
Z_p \times Z_p \ar[d]_{\alpha}\ar[r]^{\tilde{A}} & \frac{1}{p}(Z_p \times Z_p) \times V_p \ar[d]^{\mathcal{F}_p(\alpha)}\\
Z_p \ar[r]_{\phi} & \frac{1}{p}Z_p \times V_p
}
\end{equation}
 \end{proof}
\begin{rem} Someone who is curious about the formula for $\tilde{A}$ may check that 
\[ \tilde{A}(a,b) = (\alpha(T(a), r) ,T(b),[a_0+b_0]_p).\] 
Here $T$ is the \emph{tail} function, i.e. $T((\ldots , a_2,a_1,a_0)) = (\ldots, a_2, a_1)$, $\alpha$ is the $p$-adic addition and $r= (\ldots, 0, k_p(a_0 + b_0))$.
\end{rem}
\subsection{The multiplication in $\pAdic$}
Let $M: N_p \times N_p \to \frac{1}{p}(N_p \times N_p) \times V_p$ be the map such that $M(m,n) = (M_2(m,n), M_1(m,n), M_0(a,b))$ where 
\begin{itemize}
\item $M_0(m,n) = [mn]_p$, 
\item $M_1(m,n) = \left\lbrace \begin{array}{cccc} 
m & \text{if } [n]_p = 0 \\
k_p(mn) & \text{ if } [n]_p \neq 0
 \end{array} \right.$
\item $M_1(m,n) = \left\lbrace \begin{array}{cccc} 
k_p(n) & \text{if } [n]_p = 0 \\
1 & \text{ if } [n]_p \neq 0
 \end{array} \right.$
\end{itemize} 
We must check: 
\begin{prop} The map $M$ is nonexpanding. 
\end{prop}
\begin{proof} In order to show $M$ is nonexpanding, we must show that $M_0$, $M_1$ and $M_2$ are nonexpanding. It is easy to verify that $M_0$ is nonexpanding, hence left to the reader. In order to check that $M_1$ and $M_2$ are nonexpanding, we must verify that  whenever $p^i\mid m_1 -n_1$ and $p^i\mid m_2 -n_2$ where $i \geq 1$, $p^{i-1} \mid (M_k(m_1,n_1) - M_k(m_2,n_2))$ for $k \in \lbrace 1,2\rbrace$. \vspace{3pt}  \\
\textbf{Case 1.} When $[n_1]_p \neq 0$ and $[n_2]_p \neq 0$. \vspace{3pt}\\ 
In this case $M_2(m_1,n_1) - M_2(m_2, n_2) = 0$, hence satisfies the required condition. Since $p^i \mid m_1 -n_1$ and $p^i \mid m_2 -n_2$, we may write $m_1 = a_1p^i + n_1$ and $m_2= a_2p^i + n_2$. Thus, 
\[ m_1m_2 -n_1n_2  = p^{2i}a_1a_2 + p^i(a_1m_2 + a_2m_1).\] 
Thus $[m_1n_1]_p = [m_2n_2]_p$. Moreover 
\begin{eqnarray*}
M_1(m_1,n_1) -M_1(m_2,n_2) &=& k_p(m_1n_1) - k_p(m_2n_2)\\
 &=& \frac{m_1n_1 -[m_1n_1]_p}{p} - \frac{m_2n_2 -[m_2n_2]_p}{p} \\ 
 &=& \frac{m_1n_1 - m_2n_2}{p}
\end{eqnarray*}
which is divisible by $p^{i-1}$. 
\vspace{3pt}  \\
\textbf{Case 2.} Either $[n_1]_p \neq 0$ and $[n_2]_p = 0$  or $[n_1]_p = 0$ and $[n_2]_p \neq 0$ \vspace{3pt}\\ 
In this case $p$ divides one of $n_1$ and $n_2$, thus $n_1-n_2$ is not divisible by $p$ and  
\[d_{N_p \times N_p}((m_1,n_1), (m_2,n_2)) = 1.\]
Since we are working with $1$-bounded metric spaces, the functions $M_1$ and $M_2$ are trivially nonexpanding.\vspace{3pt}  \\
\textbf{Case 3.} When $[n_1]_p = 0$ and $[n_2]_p = 0$. \vspace{3pt}\\  
In this case $M_1$ is a projection on the first factor which is divisible by $p^i$ and $M_2= k_p(n)$ which is divisible by $p^{i-1}$.  
\end{proof}
\begin{proof}[Proof of Main~Theorem~\ref{main:mult}]
The diagram 
\begin{equation} \label{diag:mult}
\xymatrix{
N_p \times N_p \ar[d]_{\times}\ar[r]^{M} & \frac{1}{p}(N_p \times N_p) \times V_p \ar[d]^{\mathcal{F}_p(\times)}\\
N_p \ar[r]_{\phi_{N_p}} & \frac{1}{p}N_p \times V_p
}
\end{equation}
commutes as it can be readily checked that  
\[ (k_p(mn), [mn]_p) = (M_1(m,n)M_2(m,n), [mn]_p).\] 
Next we apply the Cauchy completion functor to the Diagram~\ref{diag:mult} to obtain the commutative diagram 
\begin{equation}
\xymatrix{
Z_p \times Z_p \ar[d]_{\mu}\ar[r]^{\tilde{M}} & \frac{1}{p}(Z_p \times Z_p) \times V_p \ar[d]^{\mathcal{F}_p(\mu)}\\
Z_p \ar[r]_{\phi} & \frac{1}{p}Z_p \times V_p
} 
\end{equation}
as desired.
\end{proof}
\begin{rem} A curious reader may verify that when the elements of $Z_p$ are represented in terms of infinite streams, we get the following formula for $\tilde{M}.$ 
\begin{itemize}
\item $\tilde{M}_0(a, b) = a_0b_0$
\item $\tilde{M}_1(a,b) = \left\lbrace \begin{array}{ccc}a & \text{ if } b_0 =0\\ T(\mu(ab))& \text{otherwise} \end{array}\right.$
\item $\tilde{M}_2(a,b) = \left\lbrace \begin{array}{ccc} T(b) & \text{if }b_0 =0 \\ 1 & \text{otherwise,}\end{array}\right.$
\end{itemize}  
where $T$ is the \emph{tail} function, i.e. $T((\ldots , a_2,a_1,a_0)) = (\ldots, a_2, a_1)$. 
\end{rem}

\end{document}